\documentclass[a4paper,12pt]{article}
\usepackage{amsmath,amsfonts,amsthm,amssymb,latexsym} 
\usepackage{graphicx}

\newcommand{\C}{{\mathbb C}}
\newcommand{\R}{{\mathbb R}}

\newcommand{\Z}{{\mathbb Z}}

\def\Mat{{\rm Mat}}
\newcommand{\G}{\Gamma}
\newcommand{\de}{\em}
\def\Span{{\rm Span}}
\def\dim{{\rm dim}}

\newtheorem{theorem}{Theorem}[section]
\newtheorem{lemma}[theorem]{Lemma}
\newtheorem{corollary}[theorem]{Corollary}
\newtheorem{proposition}[theorem]{Proposition}

\theoremstyle{definition}
\newtheorem{note}[theorem]{Note}
\newtheorem{definition}[theorem]{Definition}
\newtheorem{notation}[theorem]{Notation}

\newtheorem{rem}[theorem]{Remark}

\begin{document}
\pagestyle{plain}

\title{The quantum adjacency algebra and
\\subconstituent algebra of a graph}
\author{
Paul Terwilliger
\\[+3pt]
{\normalsize Department of Mathematics, University of Wisconsin,} \\
{\normalsize 480 Lincoln Drive, Madison, WI 53706-1388 USA} \\[+3pt]
{\normalsize email: \texttt{terwilli@math.wisc.edu}}\\[+6pt] 
\and 
Arjana \v{Z}itnik
\\[+3pt]
{\normalsize Faculty of Mathematics and Physics, University of Ljubljana,   and IMFM} \\
{\normalsize Jadranska 19, 1000 Ljubljana, Slovenia} \\[+3pt]
{\normalsize email: \texttt{Arjana.Zitnik@fmf.uni-lj.si}} 
}
\date{September 27, 2017}
\maketitle

% %- - - - - - - - - - - - - - ABSTRACT Abstract - - - - - - - - - - - -
\begin{abstract}
Let $\G$ denote a finite, undirected, connected graph, with vertex set $X$.
Fix  a vertex $x \in X$. Associated with $x$ is a certain subalgebra $T=T(x)$
of $\Mat_X(\C)$, called the subconstituent algebra.
The algebra $T$ is semisimple.
Hora and Obata introduced a certain  subalgebra $Q \subseteq T$,
called the quantum adjacency algebra. The algebra $Q$ is semisimple.
In this paper we investigate how $Q$ and $T$ are related.
In many cases $Q=T$, but this is not true in general.
To clarify this issue, we introduce the notion of 
quasi-isomorphic irreducible $T$-modules. 
We show that the following are equivalent:
(i) $Q \ne T$; 
(ii)  there exists a pair of quasi-isomorphic irreducible $T$-modules 
that have different endpoints.
To illustrate this result we consider two  examples. 
The first example concerns the Hamming graphs.
The second example concerns the bipartite dual polar graphs.  
We show that for the first example $Q=T$, and for the second
example $Q \ne T$. 

\end{abstract}

\noindent
{\bf 2010 Mathematics Subject Classification:}  
%16D70,   %  Associative rings and algebras, modules, bimodules and ideals;
         %  Structure and classification
05E15,   %  Combinatorial aspects of groups and algebras
%05E10,  %  Combinatorial aspects of representation theory 
05E30.   %  Association schemes, strongly regular graphs
%33D80.   % Connections with quantum groups, Chevalley groups, $p$-adic groups, Hecke algebras, and related topics

%--------------------------------------------------------------------------
%--------------------sec:introduction--1se-----------------------------
\section{Introduction}

In \cite{Talg} the subconstituent algebra  was introduced, and used
to investigate commutative association schemes.
Since then the subconstituent algebra has received considerable attention;
some notable papers are \cite{CaughmanWolff,Curtin,Tanaka,Go,ItoNomuraTerwilliger,
ItoTanabeTerwilliger,Lee,MacLeanTerwilliger,MiklavicTerwilliger,Morales,
Muzychuk,Pascasio,
Schrijver,
Terw1,Worawann}.

%this algebra has been linked to 
%the Lie algebra ${\mathfrak{sl}}_2$,
%the Askey-Wilson algebra,
%the $q$-Onsager algebra,
%the quantum groups $U_q({\mathfrak{sl}}_2)$, $U_q(\widehat{\mathfrak{sl}}_2)$.
In the present paper we consider the subconstituent algebra of a graph.
Let $\G$ denote a finite, undirected, connected graph, with vertex set $X$
and path-length distance function $\partial$.
Let $A \in \Mat_X(\C)$ denote the adjacency matrix of $\G$.
Fix  $x \in X$ and define
$D=\max \{\partial(x,y)\, |\,y \in X\}$.
For $0 \le i \le D$ let
$E_i^*$ denote the diagonal matrix in $\Mat_X(\C)$ 
with $(y,y)$-entry
$$
(E_i^*)_{yy}\, =\, \left\{ \begin{array}{ll}
                          1 & \mbox{if} \ \ \partial(x,y)=i,\\
                          0 & \mbox{if} \ \ \partial(x,y)\ne i\\
                           \end{array} \right. 
\ \ \ \ \ (y \in X).
$$
Let $T$ denote the subalgebra of $\Mat_X(\C)$ generated by 
$A$ and $\{E_i^*\}_{i=0}^D$. We call  $T$ the subconstituent algebra
(or Terwilliger algebra) of $\G$ with respect to  $x$ \cite[p.~380]{Talg}.
The algebra $T$ is semisimple (see Section \ref{sec:prelim}).

In \cite{HoraObata}  A. Hora and N. Obata introduced  the quantum  adjacency algebra,  
and used it to investigate quantum probability. This algebra is described as follows.
Define 
$$
L = \sum_{i=1}^{D}E_{i-1}^*AE_{i}^*, \ \ \ \ \ \ \ \
F = \sum_{i=0}^{D}E_{i}^*AE_{i}^*, \ \ \ \ \ \ \ \
R = \sum_{i=0}^{D-1}E_{i+1}^*AE_{i}^*.
$$
%We call $L, F, R$ the {\de lowering matrix}, the {\de flattening matrix}
%and the {\de raising matrix}, respectively.
%Observe that $L^t = R$ and $F^t = F$. Moreover $A = L + F + R$.
Observe that  $L,F,R \in T$. The equation $A = L + F + R$
is called the quantum decomposition of $A$ with respect to $x$ \cite[Definition 2.24]{HoraObata}.
Let $Q$ denote the subalgebra of $T$ generated by $L,F,R$. 
We call  $Q$  the  quantum adjacency algebra of $\G$
with respect to $x$ \cite[p. 78]{HoraObata}. 
The algebra $Q$ is semisimple (see Lemma \ref{lem:Qct}).

In this paper we investigate how $Q$ and $T$ are related.
In many cases $Q=T$, 
but it turns out that this is not always true.
To clarify this issue, we compare the $Q$-modules and $T$-modules.
%We mentioned  that $Q$ is a subalgebra of $T$.
Let $W$ denote a $T$-module. 
For the $T$-module $W$ 
the restriction of the $T$-action
to $Q$ turns $W$ into a $Q$-module.
Assume that the $T$-module $W$ is irreducible. 
We show that the $Q$-module $W$ is irreducible.
Let $U$ and $W$ denote irreducible $T$-modules.
Then the $Q$-modules $U,W$ are irreducible.
Assume that the $T$-modules $U,W$ are isomorphic.
By construction the $Q$-modules $U,W$ are isomorphic.
Next assume that the $T$-modules $U,W$ are not isomorphic.
%It is possible that the $Q$-modules $U,W$ are isomorphic.
We find a necessary and sufficient condition for the $Q$-modules
$U,W$ to be isomorphic.
To describe this condition, we introduce the notion of 
quasi-isomorphic irreducible $T$-modules.
The main result of the paper is that the following are equivalent:
(i) $Q \ne T$; 
(ii)  there exists a pair of quasi-isomorphic irreducible $T$-modules 
that have different endpoints.
To illustrate the main result we consider two  examples. 
The first example concerns the Hamming graphs \cite[p. 261]{BCN}. 
The second example concerns the bipartite dual polar graphs \cite[p. 274]{BCN}.  
We show that for the first example $Q=T$, and for the second
example $Q \ne T$.

The paper is organized as follows. 
In Section 2 we recall some background concerning 
semisimple algebras and their modules. 
In Section 3 we recall the subconstituent algebra $T$.
In Section 4 we consider the quantum adjacency algebra $Q$
and  its relationship to $T$.
%. 
In Section 5 we describe  a $\Z$-grading for $Q$ and $T$.
In Sections 6  and 7 we compare the irreducible $Q$-modules and $T$-modules.
In Section 8 we introduce the notion of quasi-isomorphic irreducible $T$-modules.
In Section 9 we compare the algebras $Q$ and $T$;
Theorem \ref{thm:QisT} is the main result of the paper.
In Section 10 we consider a type of irreducible $T$-module,
said to be thin.
In Section 11  we discuss the Hamming graphs and
the bipartite dual polar graphs.

%--------------------------------------------------------------------------
%--------------------sec:preliminaries-2se-----------------------------
\section{Preliminaries}
\label{sec:prelim}

In this section we review some basic facts 
concerning algebras and their modules.

Let $X$ denote a nonempty finite set.
Let  $\Mat_X(\C)$  denote the $\C$-algebra consisting of the matrices 
whose rows and columns are indexed by $X$ and
whose entries are in $\C$. 
Let  $ I \in \Mat_X(\C)$ denote the identity matrix.
Let $V = \C^X$ denote the vector space over $\C$ consisting of the column 
vectors whose coordinates are indexed by $X$ and whose entries are in $\C$.
Observe that $\Mat_X(\C)$ acts on $V$ by left multiplication.
We endow $V$ with the Hermitean inner product 
$\langle \, , \, \rangle$ that satisfies
$\langle u, v \rangle = u^t \overline{v}$ for $u, v \in V$. 
Here $t$ denotes transpose and $\overline{\phantom{v}}$  denotes complex conjugation.
%We call $V$ the {\em standard module}. % of $\Mat_X(\C)$.
%For all $y \in X$, let $\hat{y}$ denote the element of 
%$V$ with $1$ in the $y$-th coordinate and $0$ in all other coordinates. 
%
Let $W,W'$ denote 
nonempty subsets of $V$.
These subsets are said
to be {\em orthogonal} whenever $\langle w,w^\prime \rangle=0$ for all 
$w \in W$ and $w^\prime \in W^\prime$.

%%and in particular the subconstituent algebras. 
%%Note that any subalgebra of $\Mat_X(\C)$ that is closed under the 
%%conjugate-transpose map is semisimple.
%For more background on the Wedderburn theory of semisimple algebras 
%we refer the reader to \cite[Chapter IV]{CurtisReiner}.
Let $S$ denote a  subalgebra of $\Mat_X(\C)$.
By an {\de $S$-module} we mean a subspace $W$ of $V$ such that
$SW \subseteq W$. 
%%$sw \in W$ for all $s \in S $ and all $w \in W$.
We refer to $V$ itself as the {\de standard module} for $S$.
An $S$-module $W$ is said to be {\de irreducible} whenever $W$ is nonzero 
and contains no $S$-module other than $0$ and $W$. 

Let $U$ and $W$ denote  $S$-modules. 
By an {\de isomorphism of $S$-modules} from $U$ to $W$ we mean
a vector space isomorphism  $\sigma:U \to W$ such that 
$(\sigma s - s \sigma)U=0$  for all $s \in S$.
The  $S$-modules $U$ and $W$ are called {\de  isomorphic} 
%%%%%%%%%%%(or {\de  isomorphic})
whenever there exists 
an  isomorphism of $S$-modules from $U$ to $W$.

For the rest of this section, assume that $S$ is closed under
the conjugate-transpose map.
Then $S$ is semisimple  \cite[Lemma 3.4]{Talg}.
Let $W$ denote an $S$-module. 
Then its orthogonal complement $W^\perp$ is an $S$-module.
It follows that every $S$-module is an orthogonal direct sum of
irreducible $S$-modules. 
In particular, $V$ is an
 orthogonal direct sum of irreducible $S$-modules.

Let $S^{\vee}$ denote the set of isomorphism classes of irreducible $S$-modules.
The elements of $S^{\vee}$ are called {\de types}. 
For $\lambda \in S^{\vee}$ let $V_\lambda$ denote the subspace of $V$ spanned 
by the irreducible $S$-modules of type $\lambda$. 
%Call $V_\psi$ the {\de homogeneous component of $V$} of type $\psi$.
Observe that $V_\lambda$ is an $S$-module. We have 
\begin{eqnarray}                                 \label{eq:decompose}
V = \sum_{\lambda \in S^{\vee}} V_\lambda \ \ \ \ \mbox{(orthogonal direct  sum)}.
\end{eqnarray}
For  $\lambda \in S^{\vee}$ let 
$d_\lambda$ denote the
dimension of an
 irreducible $S$-module 
that has type  $\lambda$. 

\begin{proposition}                  \label{thm:algebradim}
%{\rm \cite[Chapter IV]{CurtisReiner}}
{\rm \cite[Proposition 2.2]{TomiyamaSRG}}
Let $S$ denote a subalgebra of $\Mat_X(\C)$ that
is closed under the conjugate-transpose map.
Then
$$
\dim \, S = \sum_{\lambda \in S^{\vee}} \ d_\lambda^2.
$$
\end{proposition}

The following notation will be useful.

\begin{notation}
Given subspaces $Y,Z$ of the vector space $\Mat_X(\C)$ define
$$
YZ = \Span\{yz| \ y \in Y, z \in Z \}.
$$
\end{notation}

\section{The  subconstituent algebra}
\label{sec:subconstalgebra}

In this section we recall the subconstituent algebra; 
see  %Terwilliger 
\cite{Talg} for more background information.

Throughout the paper $\G$ denotes a finite, undirected, connected
graph, without loops or multiple edges, with vertex set $X$ 
and path-length distance function $\partial$.
Denote by $A$ the matrix in $\Mat_X(\C)$ with 
$(x,y)$-entry 
$$
A_{xy}\, =\, \left\{ \begin{array}{ll}
                          1 & \mbox{if} \ \ \partial(x,y)=1,\\
                          0 & \mbox{if} \ \ \partial(x,y)\ne 1\\
                           \end{array} \right. 
\ \ \ \ \ (x,y \in X).
$$
We call $A$ the {\de adjacency matrix} of $\G$. 
Note that $A$ is real and symmetric. 
%
%--------------------Bose-Mesner------------------
Let $M$ denote the subalgebra of $\Mat_X(\C)$ generated by $A$.
We call $M$ the {\de adjacency algebra} of $\G$.

%--------------------dual Bose-Mesner-----------------
We now recall the dual adjacency algebras of $\G$ \cite[p.~378]{Talg}. 
For the rest of this paper, fix  $x \in X$. Define $D=D(x)$ by
$D=\max \{\partial(x,y)\, |\,y \in X\}$.
We call $D$ the {\de diameter of $\G$ with respect to $x$}.
For $0\le i \le D$ let $E_i^* = E_i^*(x)$ 
denote the diagonal matrix in $\Mat_X(\C)$ with $(y,y)$-entry
\begin{equation}  
\label{eq:defEi}
(E_i^*)_{yy}\, =\, \left\{ \begin{array}{ll}
                          1 & \mbox{if} \ \ \partial(x,y)=i,\\
                          0 & \mbox{if} \ \ \partial(x,y)\ne i\\
                           \end{array} \right. 
\ \ \ \ \ (y \in X).
\end{equation}
We call $E_i^*$ the {\de $i$-th dual idempotent of $\G$ with 
respect to $x$.}
The following definition is for notational convenience.
For $i \in \mathbb Z$,
\begin{equation}                               \label{eq:Eiis0}
E_i^*=0 \ \ \ \ \ {\rm unless} \ \  0 \leq i \leq D. 
\end{equation}
Observe  that 
\begin{equation}                   \label{eq:Isum}
I=\sum_{i=0}^D E_i^*
\end{equation}
and
\begin{equation}                   \label{eq:EiEj}
E_i^*E_j^*=\delta_{ij}E_i^*  \qquad \qquad  (0\le i,j\le D).
\end{equation}
The matrices $\{ E_i^*\}_{i=0}^D$ form a basis for 
a commutative subalgebra $M^*=M^*(x)$ of  $\Mat_X(\C)$.
We call $M^*$ the {\de dual adjacency algebra of $\G$
with respect to $x$}. Note that
\begin{equation}               \label{eq:Vsum}
V=\sum_{i=0}^DE_i^*V  \ \ \ \ \mbox{(orthogonal direct  sum)}.
\end{equation}

We now recall how $M$ and $M^*$ are related. 
By 
(\ref{eq:defEi}) 
and the definition of $A$,
\begin{equation}                                  \label{combcoeff}
E_i^* A E_j^* = 0\ \ \ \ \ \mbox{if} \ \  |i-j|>1, 
\qquad \quad (0 \leq i,j \leq D).
\end{equation}

Let $T = T (x)$ denote the subalgebra of $\Mat_X(\C)$ generated by 
$M$ and $M^*$. We call  $T$ the {\de subconstituent algebra 
of $\G$ with respect to  $x$} \cite[p.~380]{Talg}.
The algebra $T$ is often called the {\de Terwilliger algebra} \cite{Caughman,Curtin}.
%%Note that $T$ is generated by  $A$ and $\lbrace E_i^*\rbrace_{i=0}^D$.
%%These generators are real and symmetric, so
%%$T$ is closed under the conjugate-transpose map.
%%Hence $T$ is semisimple  \cite[Lemma 3.4]{Talg}.

%-----------------------  T-modules----------------------------------

Observe that $T$ is generated by  $A$ and $\lbrace E_i^*\rbrace_{i=0}^D$.
These generators are real and symmetric, so
$T$ is closed under the conjugate-transpose map. Therefore our
discussion in Section \ref{sec:prelim} applies to $T$. 
By Proposition \ref{thm:algebradim},
\begin{equation}                                  \label{eq:dimT1}
\dim \, T = \sum_{\lambda \in T^\vee} d_\lambda^2.
\end{equation}

Let $W$ denote an irreducible $T$-module. 
Observe that $W$ is an orthogonal direct sum of the nonzero subspaces among 
$E_0^*W,\dots, E_D^*W$. 
Define the \emph{endpoint} $r=r(W)$ by
$$
r = \min\{i \, \vert \ 0 \le i \le D, E_i^*W \ne 0\}.
$$ 
Define the {\de diameter} $d=d(W)$  by 
$$
d=\big|\{i \, \vert \ 0 \le i \le D, E_i^*W \ne 0\}\big|-1.
$$ 
Using the idea from
\cite[Lemma 3.9(ii)]{Talg}  we have
$E_i^*W \ne 0$  if and only if $r \le i \le r+d$  $(0 \le i \le D)$.
By the above comments,
\begin{equation}                                       \label{eq:Wsum}
W = \sum_{i=0}^{d}\ E_{r+i}^*W     \ \ \ \ \mbox{(orthogonal direct  sum)}.
\end{equation}
Note that isomorphic irreducible $T$-modules
have the same endpoint and the same diameter.

%%--------------------------------------------------------------------------
%\begin{lemma}{\rm \cite[Lemma 3.4, Lemma 3.9, Lemma 3.12]{Talg}.}       \label{thin}
%Let $W$ denote an irreducible $T$-module with endpoint $\rho$, 
%dual endpoint $\tau$, and diameter $d$. Then $\rho,\tau,d$ are nonnegative
%integers such that $\rho+d \le D$ and $\tau + d \le D$. Moreover the following
%{\rm (i)--(iv)} hold.
%\begin{itemize}
%\item[{\rm (i)}] $E_i^*W \ne 0$ if and only if $\rho \le i \le \rho+d,$ \ \ \ $(0 \le i \le D)$.
%\item[{\rm (ii)}] $W=\sum_{h=0}^d \, E_{\rho+h}^*W$ \ \ \  \mbox{\rm (orthogonal direct sum)}.
%\item[{\rm (iii)}] $E_iW \ne 0$ if and only if $\tau \le i \le \tau+d,$ \ \ \ $(0 \le i \le D)$.
%\item[{\rm (iv)}] $W=\sum_{h=0}^d \, E_{\tau+h}W$ \ \ \ \mbox{\rm (orthogonal direct sum)}.
%\end{itemize}
%\end{lemma}
%%--------------------------------------------------------------------------

%Assume that the $T$-modules $W$ and $W'$ are isomorphic.
%Then they have the same endpoint and diameter  \cite[Lemma 4.3]{Talg2}. 
%Now assume that the  $T$-modules $W$ and $W'$ are not isomorphic.
%Then $W$ and $W'$ are orthogonal \cite[Lemma 3.3]{Curtin}.

%----------------------------- 4se-----------------------------------
\section{The quantum adjacency algebra}

In this section we recall from  \cite{HoraObata}
the  quantum adjacency algebra.
%We obtain some results concerning how these algebras are related to $T$.

\begin{definition}  
\label{def:LFR}
{\rm \cite[p.~10]{DickieTerwilliger}} 
Define  %matrices $L=L(x)$, $F=F(x)$, $R=R(x)$ in  $\Mat_X(\C)$  by
\begin{align*}
L = \sum_{i=1}^{D}E_{i-1}^*AE_{i}^*,
\qquad \quad
F = \sum_{i=0}^{D}E_{i}^*AE_{i}^*,
\qquad \quad
R = \sum_{i=0}^{D-1}E_{i+1}^*AE_{i}^*.
\end{align*}
%We call $L, F, R$ the {\de lowering matrix}, the {\de flattening matrix}
%and the {\de raising matrix}, respectively.
\end{definition}

\begin{lemma} The matrices $L,F,R$ are contained in $T$.
\end{lemma}
\begin{proof} 
By Definition \ref{def:LFR}.
\end{proof}

\begin{lemma}
\label{lem:LRFt}
%{\rm \cite[p.~x]{Talg}} 
The matrices $L,F,R$ are real. Moreover $L^t = R$ and $F^t = F$. 
\end{lemma}
\begin{proof}
The matrix $A$ is real and symmetric. 
For $0 \le i \le D$ the matrix $E_i^*$ is real and diagonal.
The results follow by Definition \ref{def:LFR}. 
\end{proof}

\begin{lemma}      \label{lem:LFRsubspace}
For $0 \le i \le D$ the matrices $L,F,R$ act on  $E_i^*V$ in the following way:
$$
LE_i^*V \subseteq E_{i-1}^*V, \ \ \ \ \ \ \ \
FE_i^*V \subseteq E_{i}^*V, \ \ \ \ \ \ \ \
RE_i^*V \subseteq E_{i+1}^*V.
$$
\end{lemma}
\begin{proof} 
%Since $LV \subseteq V$, $RV \subseteq V$ and $FV \subseteq V$, the result 
By 
\eqref{eq:EiEj}
and
Definition \ref{def:LFR}.
%%%%%%%%see also Lemma \ref{lemma:pushEi}.
% By Lemma \ref{lemma:pushEi}.
\end{proof}

\begin{lemma}                        \label{lemma:AisLFR}
{\rm \cite[p.~10]{DickieTerwilliger}} 
We have
\begin{equation}
A = L + F + R.
\label{eq:LFR}
\end{equation}
\end{lemma}
\begin{proof}
Multiply $A$ on the left and  right by $I$. Evaluate the result
using \eqref{eq:Isum} and \eqref{combcoeff}. Line
(\ref{eq:LFR})
follows
by Definition \ref{def:LFR}. 
\end{proof}

\begin{note}
In \cite{HoraObata} the matrices $L,F,R$ are called
$A^-, A^0, A^+$ respectively.
In \cite[Definition 2.24]{HoraObata}
the  equation
(\ref{eq:LFR}) 
is called the {\de quantum decomposition of $A$ with respect to $x$}. 
\end{note}

\begin{definition}
\label{def:Q}
\rm \cite[p. 78]{HoraObata}.
Let $Q=Q(x)$ denote the subalgebra of $T$ generated by 
$L,F,R$. We call $Q$ the {\de quantum adjacency algebra of $\G$
with respect to $x$}. 
\end{definition}

\begin{note}
The algebra $Q$ is called 
$\tilde{\cal A}$ in
\cite{HoraObata}.
\end{note}

\begin{lemma} 
\label{lem:Qct}
The algebra $Q$ is closed under the
conjugate-transpose map.
\end{lemma}
\begin{proof}
By Lemma
\ref{lem:LRFt}
and
Definition \ref{def:Q}.
\end{proof}

\begin{lemma} 
\label{lem:dimQ}
We have
\begin{equation}                                  \label{eq:dimQ}
\dim \, Q = \sum_{\mu \in Q^\vee} d_{\mu}^2.
\end{equation}
\end{lemma}
\begin{proof}
By Proposition \ref{thm:algebradim}.
\end{proof}

%\begin{lemma}
%$Q$ is a subalgebra of $T$.
%\end{lemma}

As we will see, in some cases $Q=T$,
and in other cases  $Q \ne T$.
We now consider how $Q$ is related to $T$ in general.

\begin{lemma}                  \label{lemma:pushEi}
We have
\begin{eqnarray*}
&&LE_i^* = E_{i-1}^*L \qquad (1 \leq i \leq D), 
\qquad LE^*_0 = 0, \qquad E^*_D L = 0, \\
&&
FE_i^* = E_i^*F \qquad (0 \leq i \leq D),
\\
&&
RE^*_{i-1} = E^*_i R \qquad (1 \leq i \leq D),
\qquad 
R E^*_D = 0 ,\qquad 
E^*_0 R = 0.
\end{eqnarray*}
\end{lemma}
\begin{proof}
By
\eqref{eq:EiEj}
and
Definition \ref{def:LFR}.
\end{proof}

%%moved
%%\begin{notation}
%%Given subspaces $Y,Z$ of $\Mat_X(\C)$ define
%%$$
%%YZ = \Span\{yz| \ y \in Y, z \in Z \}.
%%$$
%%\end{notation}

\begin{corollary}                            \label{cor:LMML}
We have
$$
LM^*=M^*L, \qquad \qquad
FM^*=M^*F, \qquad \qquad
RM^*=M^*R.
$$
\end{corollary}
\begin{proof}
By Lemma \ref{lemma:pushEi} and since $M^*$ is spanned by $\{E_i^*\}_{i=0}^D$.
\end{proof}

\begin{proposition}                                \label{thm:QM}
We have
$$
QM^*=T=M^*Q.
$$
\end{proposition}
\begin{proof}
By Corollary \ref{cor:LMML} and since $Q$ is generated
by $L,F,R$, we see that $QM^*=M^*Q$.
This common value is a subalgebra of
$T$ that contains
$M^*$ and $Q$. The algebra $T$ is generated by 
$M^*,Q$. The result follows.
\end{proof}

%---------------------------5se----------------------------------------
\section{A $\Z$-grading for $Q$ and $T$}
\label{section:Zgrading}

In this section we describe a  $\Z$-grading for $Q$ and $T$.
These $\Z$-gradings will be used to compare
the $Q$-modules and $T$-modules.

\begin{lemma}                        \label{lemma:Tsum0} 
The following is a direct sum of vector spaces:
\begin{equation}
\label{eq:Texpand}
     T = \sum_{i=0}^D \sum_{j=0}^D E^*_i T E^*_j.
\end{equation}
\end{lemma}
\begin{proof}
Multiply $T$ on the left and  right by $I$, and use
\eqref{eq:Isum} to obtain \eqref{eq:Texpand}.
The sum \eqref{eq:Texpand} is direct by \eqref{eq:Isum} and \eqref{eq:EiEj}.
\end{proof}
\begin{definition}                                \label{def:defTn}
For  $n \in \Z$ define a subspace $T_n \subseteq T$ by 
$$                                                
T_n= \sum_{i \in \Z} E_{i+n}^* T E_{i}^*.
$$
\end{definition}

\begin{lemma}                        \label{lemma:Tsum} 
The following is a direct sum of vector spaces:
\begin{equation}                   \label{eq:gradeT}
T = \sum_{n\in \Z} T_n.
\end{equation}
\end{lemma}
\begin{proof}
Combine Lemma \ref{lemma:Tsum0} and Definition \ref{def:defTn}.
%Using \eqref{eq:Texpand} and \eqref{eq:Eiis0} we see that 
%$$ T = \sum_{i=0}^D \sum_{j=0}^D E^*_i T E^*_j
%=  \sum_{n \in \Z}\sum_{i \in \Z} \, E_{i-n}^*T E_{i}^*
%= \sum_{n \in \Z} T_n.
%$$
%The sum   \eqref{eq:gradeT} is direct 
%since the sum \eqref{eq:Texpand} is direct.
\end{proof}

\begin{lemma}                        \label{lemma:Tn0} 
For $n \in \mathbb Z$ we  have $T_n =0$ unless $-D \leq n \leq D$.
\end{lemma}
\begin{proof}
By 
\eqref{eq:Eiis0}
and
Definition \ref{def:defTn}.
\end{proof}

\begin{lemma} 
\label{lem:TnAction}
For $n \in \Z$ and $S  \in T$ the following are equivalent:
\begin{itemize}
\item[{\rm (i)}]  $S \in T_n$;
\item[{\rm (ii)}]  $S E^*_i V \subseteq E^*_{i+n}V$ for $i \in \Z$.
\end{itemize}
\end{lemma} 
\begin{proof} 
(i) $\Rightarrow$ (ii) By Definition \ref{def:defTn} and \eqref{eq:EiEj}.\\
(ii) $\Rightarrow$ (i) By Lemma   \ref{lemma:Tsum} and \eqref{eq:EiEj},
 \eqref{eq:Vsum}.
%since the sum $V=\sum_{i=0}^D\, E_i^*V$ is direct.
\end{proof}

\begin{lemma}              \label{lemma:Tncontains}
We have $M^* \subseteq T_0$.
 Moreover
\begin{equation}
L \in T_{-1}, \qquad \qquad
F \in T_0,  \qquad \qquad
R \in T_{1}.
\label{eq:WhereLFR}
\end{equation}
\end{lemma}
\begin{proof}
By
\eqref{eq:EiEj}
and
Definition \ref{def:defTn},
$T_0$ contains $E_j^*$ for $0 \le j \le D$.
Therefore $T_0$ contains $M^*$.
Line (\ref{eq:WhereLFR}) 
follows from  Lemmas \ref{lem:LFRsubspace}
and \ref{lem:TnAction}.
%Definition \ref{def:defTn}.
\end{proof}

\begin{lemma}      \label{lemma:TnEi} 
For $n \in \Z$ and $S \in T$ the following are equivalent:
\begin{itemize}
\item[{\rm (i)}]  $S \in T_n$;
\item[{\rm (ii)}]  $SE^*_i = E^*_{i+n} S$ for $i \in \Z$.
\end{itemize}
Suppose {\rm (i)}, {\rm (ii)} hold. Then $SE^*_i =  E^*_{i+n}S E^*_i = E^*_{i+n}S $
for $i \in \Z$.
\end{lemma}
\begin{proof}
(i) $\Rightarrow$ (ii) By Definition \ref{def:defTn} and \eqref{eq:EiEj}.\\
(ii) $\Rightarrow$ (i) By Lemma   \ref{lemma:Tsum0},
\eqref{eq:EiEj}, and Definition \ref{def:defTn}.\\
Suppose {\rm (i)}, {\rm (ii)} hold. 
Then the last assertion holds by Definition \ref{def:defTn} and \eqref{eq:EiEj}.
\end{proof}

\begin{lemma}        \label{lemma:TnTm}
We have 
%\begin{equation}                      \label{eq:gradeT}
%T =  \sum_{n \in \Z} \, T_n \ \ \ \ \ \ \mbox{\rm (direct sum)}.
%\end{equation}
%Moreover, 
\begin{equation}                      \label{eq:TnTm}
T_nT_m \subseteq T_{n+m} \qquad \quad (n,m \in \Z).
\end{equation}
\end{lemma}
\begin{proof}
Use Definition
  \ref{def:defTn} and \eqref{eq:EiEj}.
\end{proof}

\begin{note}
By Lemmas \ref{lemma:Tsum}  and  \ref{lemma:TnTm},
the sequence $\lbrace T_n\rbrace_{n \in \mathbb Z}$ is a $\Z$-grading of $T$.
\end{note}

\begin{lemma} The subspace $T_0$ is a subalgebra of $T$.
\end{lemma}
\begin{proof} In 
   (\ref{eq:TnTm})
   set $m=n=0$.
   \end{proof}

\begin{definition}                                \label{eq:defQn}
For  $n \in \Z$ define  
$
Q_n=Q \cap T_n.
$
\end{definition}

\begin{lemma}    
For $n \in \mathbb Z$ we have
$Q_n =0$ unless $-D \leq n \leq D$.
\end{lemma}
\begin{proof} 
By Lemma  \ref{lemma:Tn0} and Definition \ref{eq:defQn}.
\end{proof}

\begin{lemma} 
\label{lem:QnAction}
For $n \in \Z$ and $S \in Q$ the following are equivalent:
\begin{itemize}
\item[{\rm (i)}]   $S \in Q_n$;
\item[{\rm (ii)}]  $S E^*_iV \subseteq E^*_{i+n}V$ for $i \in \Z$.
\end{itemize}
\end{lemma}
\begin{proof} By
Lemma
\ref{lem:TnAction}
and Definition
   \ref{eq:defQn}.
\end{proof}

\begin{lemma}
We have
$$
I \in Q_0, \ \ \ \ \ \ \ \
L \in Q_{-1}, \ \ \ \ \ \ \ \
F \in Q_0, \ \ \ \ \ \ \ \ 
R \in Q_{1}.
$$
\end{lemma}
\begin{proof} 
By Lemma  \ref{lemma:Tncontains} and Definiton \ref{eq:defQn}.
\end{proof}

\begin{proposition}      \label{prop:gradeQ}
The following is a direct sum of vector spaces:
\begin{equation}                      \label{eq:gradeQ}
Q =  \sum_{n \in \Z} \, Q_n.
\end{equation}
Moreover, 
\begin{equation}                                \label{eq:QnQm}
Q_nQ_m \subseteq Q_{n+m} \qquad \qquad (n,m \in \mathbb Z).
\end{equation}    
\end{proposition}
\begin{proof}
Line \eqref{eq:gradeQ} follows from  Lemmas   
\ref{lemma:Tsum}, \ref{lemma:Tncontains}
and Definition \ref{eq:defQn}.
Line \eqref{eq:QnQm} follows from  \eqref{eq:TnTm} 
and Definition \ref{eq:defQn}.
\end{proof}

\begin{note}
By Proposition \ref{prop:gradeQ}, the sequence
$\lbrace Q_n \rbrace_{n \in \Z}$ 
is a $\Z$-grading of $Q$.
\end{note}

\begin{lemma} The subspace $Q_0$ is a subalgebra of
$Q$.
\end{lemma}
\begin{proof} 
In (\ref{eq:QnQm}) set $m=n=0$.
\end{proof}

In general the  $\lbrace E^*_i\rbrace_{i=0}^D$ are not
contained in $Q$. Nevertheless we have the following.

\begin{lemma}                     \label{lemma:gradeQ} 
For $n \in \Z$ and $S \in Q$ the following are equivalent:
\begin{itemize}
\item[{\rm (i)}]  $S \in Q_n$;
\item[{\rm (ii)}]  $SE^*_i = E^*_{i+n} S$ for $i \in \Z$.
\end{itemize}
Suppose {\rm (i)}, {\rm (ii)} hold. 
Then $SE^*_i = E^*_{i+n}S E^*_i = E^*_{i+n}S $ for $i \in \Z$.
\end{lemma}
\begin{proof}
By Lemma \ref{lemma:TnEi}  and Definition   \ref{eq:defQn}.

\end{proof}

%%%\label{lem:QnAction}
%\begin{corollary}                     \label{cor:gradeQ} 
%Let $W$ denote a $T$-module. Then for $i,n \in \Z$ we have
%$$
%Q_n E_i^*W \subseteq E_{i-n}^*W.  
%$$
%\end{corollary}

%----------------------------6se---------------------------------------

\section{Irreducible $T$-modules and $Q$-modules}
\label{section:Qmodules}

Recall that $Q$ is a subalgebra of $T$.
Let $W$ denote a $T$-module. 
For the $T$-module $W$ 
the restriction of the $T$-action
to $Q$ turns $W$ into a $Q$-module.
Assume that
the $T$-module $W$ is irreducible. In this section
we show that the $Q$-module $W$ is irreducible.

\begin{lemma}
\label{lem:Wgen}
Let $W$ denote an irreducible $T$-module.
Then for $0 \ne v \in W$ we have $Tv=W$.
\end{lemma}
\begin{proof} Since  $Tv$ is a nonzero 
$T$-submodule of $W$.
\end{proof}

\begin{lemma}                        \label{lemma:generateQ}
Let $W$ denote an irreducible $T$-module.
Pick a nonzero $v \in W$ that is a common eigenvector for $M^*$.
Then  $W=Qv$.
\end{lemma}
\begin{proof}
By assumption $M^*v = \C v$. Using
Proposition
 \ref{thm:QM}
and
Lemma
\ref{lem:Wgen} we obtain
$W=Tv=QM^*v=Qv$.
\end{proof}

\begin{proposition}             \label{thm:irreducible}
Let $W$ denote an irreducible $T$-module. 
Then the $Q$-module  $W$ is  irreducible.
\end{proposition}
\begin{proof}
Let $r$ and $d$ denote the endpoint and diameter of $W$,
respectively.
Let $W^\prime$ denote a nonzero $Q$-submodule of $W$. 
We show that $W=W^\prime$.
Pick $0 \not=z \in W^\prime$. 
By \eqref{eq:Isum} we have $z=\sum_{i=0}^d \, E_{r+i}^*z$. 
Define $j=\max\{i|0 \le i \le d,\; E_{r+i}^*z \ne 0 \}$.
Thus  $z=\sum_{i=0}^j \, E_{r+i}^*z$ 
%Define $v=E_{r+j}^*z$. Then $v$ is nonzero by definition of $j$.
and $E_{r+j}^*z \not=0$.
We have
\begin{equation}                                            \label{eq:Qjz}
Q_{-j}z =  Q_{-j} \sum_{i=0}^j \, E_{r+i}^*z 
%     =    \sum_{i=0}^j \, Q_j E_{r+i}^*z
%    = Q_jv
     = Q_{-j} E_{r+j}^*z
     = E_r^*Q_{-j} E_{r+j}^*z
     = E_r^*Q E_{r+j}^*z
     =E_r^*W.
\end{equation}
In the above line, the second equality holds by 
%%%Corollary \ref{cor:gradeQ},
Lemma \ref{lem:QnAction}, the definition of $r$,
and since $W$ is a $Q$-module.
The third equality holds by Lemma \ref{lemma:gradeQ}.
The fourth equality holds by  \eqref{eq:EiEj}, \eqref{eq:gradeQ}, and  Definition \ref{eq:defQn}.
The last equality follows by Lemma \ref{lemma:generateQ} 
and since  $E_{r+j}^*z$ is a common eigenvector for $M^*$.
%By construction $W'$ is a $Q$-module, so
%$Qz \subseteq W^\prime$.
By 
Lemma \ref{lemma:generateQ} we have
$W=QE_r^*W$.
By this and  \eqref{eq:Qjz},
$$W=QE_r^*W =QQ_{-j}z\subseteq Qz \subseteq W^\prime.$$
Therefore $W=W^\prime$.
\end{proof}

Let $U$ and $W$ denote irreducible $T$-modules.
We just saw that the $Q$-modules $U,W$ are irreducible.
Assume for the moment 
that the $T$-modules $U,W$ are isomorphic.
Then the $Q$-modules $U,W$ are isomorphic.
Next assume that the $T$-modules $U,W$ are not isomorphic.
It is possible that the $Q$-modules $U,W$ are isomorphic.
%This issue will be discussed in Section \ref{section:Qmodulesiso}.
We now explain this point in more detail.
By \eqref{eq:decompose} the following sums are direct:
\begin{equation*}                                 \label{eq:decomposeQT}
V = \sum_{\lambda \in T^{\vee}} V_\lambda,\ \ \ \  \ \ \ \  \ \ \ \ 
V = \sum_{\mu \in Q^{\vee}} V_\mu.
\end{equation*}
By this and Proposition  \ref{thm:irreducible},  the inclusion map $Q \to T$ 
induces a surjective map $\psi : T^\vee \to Q^\vee$
such that for $\mu \in Q^\vee$,
\begin{equation*}                    \label{eq:Vmu}
  V_\mu = \sum_{\substack{\lambda \in T^\vee \\[0.7mm] \psi(\lambda)=\mu}} V_\lambda.
\end{equation*}
For $\mu \in Q^\vee$ define
\begin{equation}                        \label{eq:mmu}
m_\mu = \big|\{ \lambda \in T^\vee \, | \psi(\lambda)=\mu\}\big|.
\end{equation}                       
Note that $m_\mu$ is a positive integer. 

\begin{lemma}                        \label{lemma:dimT}
We have
\begin{equation}                                  \label{eq:dimTT}
\dim \, T  = \sum_{\mu\in Q^\vee} m_\mu d_\mu^2.
\end{equation}
\end{lemma}
\begin{proof}
For $\lambda \in T^\vee$ and $\mu \in Q^\vee$ such that 
$\psi(\lambda)=\mu$, we have $d_\lambda = d_\mu$.
Using this fact and  \eqref{eq:mmu}, we evaluate equation 
  \eqref{eq:dimT1}. The result follows.
\end{proof}

%-------------------------7se------------------------------------------
\section{Irreducible $T$-modules and $Q$-modules, cont.}
\label{section:Qmodulescontinued}

In this section we describe how an irreducible $T$-module looks
when viewed as a $Q$-module.

\begin{lemma}                     \label{lemma:QW1}
Let  $W$ denote an irreducible $T$-module 
with  endpoint $r$ and diameter $d$. 
For $i \in \Z$ and $r \le j \le r+d$,
\begin{equation}               \label{eq:QW1}
Q_i E_j^* W = E_{i+j}^*W.
\end{equation}
\end{lemma}
\begin{proof}
%Referring to  \eqref{eq:QW1}, the inclusion  $\subseteq$
%s by Lemma \ref{lem:QnAction}.
%%
%Now consider the inclusion  $\supseteq$.
By Lemma \ref{lemma:generateQ}, $W=QE_j^*W$.
By this and Lemma  \ref{lemma:gradeQ},
$$E_{i+j}^*W = E_{i+j}^*QE_j^*W = Q_iE_j^*W.$$
\end{proof}

\begin{lemma}                     \label{lemma:QW2}
Let  $W$ denote an irreducible $T$-module
with  endpoint $r$ and diameter $d$. 
Then for  $0 \le i \le d$,
$$          % \label{eq:QW2}
Q_i W = \sum_{\ell = i}^d \ E_{r + \ell}^*W,  \ \ \ \ \ \ \ \ \ \ \ \ \
Q_{-i} W = \sum_{\ell =0}^{d-i} \ E_{r + \ell}^*W.
$$
Moreover for $i \ge d+1$,
$$Q_i W = 0,  \ \ \ \ \ \ \ \ \ \ \ \ \  Q_{-i} W =0.$$
\end{lemma}
\begin{proof}
Use \eqref{eq:Wsum} and Lemma \ref{lemma:QW1}.
\end{proof}

\begin{corollary}                     \label{cor:QW4}
Let  $W$ denote an irreducible $T$-module
with endpoint $r$ and diameter $d$. 
Then 
$$Q_d W = E_{r+d}^*W, \ \ \ \ \ \ \ \ \ \ \ \ \  Q_{-d} W = E_{r}^*W.$$
\end{corollary}
\begin{proof}
Set $i=d$ in Lemma   \ref{lemma:QW2}.
\end{proof}

\begin{corollary}                     \label{cor:QW3}
Let  $W$ denote an irreducible $T$-module. 
Then the diameter $d$ of $W$ is given by
$$
d =  \max\{i \, |  \ 0 \leq i \leq D, Q_i W \ne 0\}.
$$
\end{corollary}
\begin{proof}
By 
Corollary \ref{cor:QW4}
and the last assertion of 
Lemma \ref{lemma:QW2}.
\end{proof}

\begin{lemma}          \label{lemma:Qisodiam}
Let  $U$ and $W$ denote  irreducible $T$-modules 
that are isomorphic as $Q$-modules. 
Then they have the same diameter.
\end{lemma}
\begin{proof}
By Corollary   \ref{cor:QW3}.
\end{proof}

\begin{proposition}                              \label{prop:quasiiso3}
Let $U$ and $W$ denote irreducible $T$-modules.
Assume that there exists an isomorphism of $Q$-modules $\sigma: U \to W$.
Then
$
\sigma E_{r+i}^*U = E_{r^\prime+i}^* W
$
for $0 \le i \le d$. Here $d=d(U)=d(W)$ and $r=r(U)$,  $r^\prime = r(W)$.
\end{proposition}
\begin{proof}   
Observe that
$$
\sigma E_{r+i}^*U =  \sigma Q_{i}E_r^*U
                    = \sigma Q_{i}Q_{-d}U
                    = Q_{i}Q_{-d} \sigma  U
                    = Q_{i}Q_{-d} W
                    = Q_{i}E_{r^\prime}^*W
                   = E_{r^\prime+i}^*W.
$$
\end{proof}

%-------------------------8se------------------------------------------
\section{Quasi-isomorphisms of  $T$-modules}
\label{section:Qmodulesiso}

In this section we compare $Q$-module isomorphisms 
and $T$-module isomorphisms. To do this, we introduce
the notion of a quasi-isomorphism of $T$-modules.

\begin{definition}                              \label{def:quasiiso}
Let $U$ and $W$ denote irreducible $T$-modules.
By a {\de quasi-isomorphism of $T$-modules from $U$ to $W$}  %with shift $n$},
we mean a $\C$-linear bijection
$
\sigma: U \to W
$
such that on $U$,
\begin{equation}                \label{eq:quasiiso1}
\sigma L= L \sigma, \ \ \ \ \ \ \ \
\sigma F= F \sigma, \ \ \ \ \ \ \ \
\sigma R= R \sigma 
\end{equation}
and 
\begin{equation}                  \label{eq:quasiiso2}
\sigma E_i^* = E_{i+n}^* \sigma \ \ \ \ \ \ \ \  i \in \Z,
\end{equation}
where $n=r(W)-r(U)$.
\end{definition}

\begin{lemma}
Let $U$ and $W$ denote irreducible $T$-modules.
For a  $\C$-linear map $\sigma: U \to W$, the following are equivalent:
\begin{itemize}
\item[{\rm (i)}] $\sigma$ is a quasi-isomorphism of $T$-modules from $U$ to $W$;
           %with shift $n$;
\item[{\rm (ii)}] $\sigma^{-1}$ is a quasi-isomorphism of $T$-modules from
          $W$ to $U$.   % with shift $-n$.
\end{itemize}
\end{lemma}
\begin{proof}
Use Definition \ref{def:quasiiso}.
%${\rm (i)} \Rightarrow {\rm (ii)}$
%Definition
%\ref{def:quasiiso} the map
%$\sigma: U\to W$ is a 
% $\C$-linear bijection.
%Note that $\sigma^{-1}$ commutes with each of $L,F,R$ since
%$\sigma$ commutes with each of $L,F,R$.
%Pick $i \in \mathbb Z$. Let $n=r(W)-r(U)$. We show that
%$\sigma^{-1} E_i^* = E_{i-n}^* \sigma^{-1}$ holds on $W$.
%Write $j=i-n$ and note that
% $\sigma E_j^*=E_{j+n}^* \sigma$ holds on $U$.
%Therefore
%\begin{align*}
%0&=\sigma^{-1} (E^*_{j+n}\sigma -\sigma E_{j}^*)U \\
%  &=(\sigma^{-1} E^*_{j+n} - E^*_j\sigma^{-1})\sigma U \\
% &=(\sigma^{-1}E^*_i- E^*_{i-n}\sigma^{-1})W.
%\end{align*}  
%We have shown that $\sigma^{-1}$ is a quasi-isomorphism of $T$-modules
%from $W$ to $U$ with shift $-n$.
%\\
%\noindent ${\rm (ii)} \Rightarrow {\rm (i)}$
%In the proof of
% ${\rm (i)} \Rightarrow {\rm (ii)}$,
%swap the roles of $U,W$ and
% replace 
%$\sigma$ by $\sigma^{-1}$.
\end{proof}

\begin{definition}                              \label{def:quasiiso2}
Irreducible $T$-modules $U$ and $W$ are called {\de quasi-isomorphic} % with shift n} 
whenever there exists a quasi-isomorphism of $T$-modules
from $U$ to $W$. % with shift $n$.
\end{definition}

We make two observations.

\begin{lemma}                     \label{lemma:Tiso}
Let $U$ and $W$ denote irreducible $T$-modules with the same endpoint.
Then for a $\C$-linear map $\sigma: U \to W$
the following are equivalent:
\begin{itemize}
\item[{\rm (i)}] $\sigma$ is a quasi-isomorphism of $T$-modules from
            $U$ to $W$;   % with shift zero;
\item[{\rm (ii)}] $\sigma$ is an isomorphism of $T$-modules from
            $U$ to $W$.
\end{itemize}
\end{lemma}
\begin{proof}
Set $n=0$ in Definition  \ref{def:quasiiso}.
\end{proof}

\begin{corollary}                     \label{corr:Tiso}
For irreducible $T$-modules $U,W$ the following are equivalent:
\begin{itemize}
\item[{\rm (i)}] the $T$-modules $U,W$ are quasi-isomorphic %with shift zero;
                       and have the same endpoint;
\item[{\rm (ii)}] the $T$-modules $U,W$ are isomorphic.
\end{itemize}
\end{corollary}
\begin{proof}
By Lemma  \ref{lemma:Tiso} and since isomorphic irreducible $T$-modules
have the same endpoint.
\end{proof}

\begin{proposition}             \label{thm:nonisomorphic}
Let $U$ and $W$ denote irreducible $T$-modules.
Then for a $\C$-linear map $\sigma: U \to W$
the following are equivalent:
\begin{itemize}
\item[{\rm (i)}] $\sigma$ is an isomorphism of $Q$-modules from $U$ to $W$;
\item[{\rm (ii)}] $\sigma$ is a quasi-isomorphism of $T$-modules from
$U$ to $W$.
\end{itemize}            
\end{proposition}
\begin{proof}
${\rm (i)} \Rightarrow {\rm (ii)}$ 
We check that $\sigma$ satisfies the conditions in Definition  \ref{def:quasiiso}.
The map $\sigma$ satisfies  \eqref{eq:quasiiso1} since 
$\sigma$ is an isomorphism of $Q$-modules.
The map $\sigma$ satisfies  \eqref{eq:quasiiso2}  by Proposition \ref{prop:quasiiso3}.
\smallskip

\noindent
${\rm (ii)} \Rightarrow {\rm (i)}$
%The map $\sigma$ is an isomorphism of $Q$-modules by \eqref{eq:quasiiso1}.
Use \eqref{eq:quasiiso1}.
\end{proof}

\begin{corollary}             \label{cor:nonisomorphic}
For irreducible $T$-modules $U,W$ the following are equivalent:
\begin{itemize}
\item[{\rm (i)}] the $Q$-modules $U,W$ are isomorphic;
\item[{\rm (ii)}] the $T$-modules $U,W$ are quasi-isomorphic.
\end{itemize}            
\end{corollary}
\begin{proof}
By Proposition  \ref{thm:nonisomorphic}.
\end{proof}

%----------------------------- 9se-----------------------------------
\section{Comparing $Q$ and $T$}

Recall that $Q$ is a subalgebra of $T$. 
In this section we consider 
when are these two algebras equal.

\begin{theorem}            \label{thm:QisT}
The following {\rm (i)}--{\rm (iv)} are equivalent:
\begin{itemize}
\item[{\rm (i)}] $Q\not=T$;
\item[{\rm (ii)}] $Q \subsetneq T$;
\item[{\rm (iii)}] there exists a pair of nonisomorphic irreducible $T$-modules 
            that are isomorphic as $Q$-modules;
\item[{\rm (iv)}] there exists a pair of quasi-isomorphic irreducible $T$-modules 
            that have different endpoints.
\end{itemize}            
\end{theorem}

\begin{proof}
${\rm (i)} \Leftrightarrow {\rm (ii)}$
By construction $Q \subseteq T$.

\noindent
${\rm (ii)} \Leftrightarrow {\rm (iii)}$
Assertion (iii) means that there exists  $\mu \in Q^\vee$ such that $m_\mu>1$.
In this light, compare   \eqref{eq:dimQ} and \eqref{eq:dimTT}.

\noindent
${\rm (iii)} \Leftrightarrow {\rm (iv)}$
By Corollaries    \ref{corr:Tiso}  and \ref{cor:nonisomorphic}.
\end{proof}

%-----------------------------10se-----------------------------------
\section{Thin irreducible $T$-modules}

In Sections \ref{section:Qmodules}--\ref{section:Qmodulesiso}
we considered the irreducible $T$-modules. In this section we consider
a special type of irreducible $T$-module, said to be thin.

\begin{definition}                              \label{def:thinmodules}
An irreducible  $T$-module $W$ is  called {\de thin}  whenever
$\dim \, E_i^*W \le 1$ for $0 \le i \le D$. 
\end{definition}

\begin{lemma}    \label{lemma:thin1}
Let $W$ denote a thin irreducible $T$-module with endpoint $r$ 
and diameter $d$. Then 
\begin{equation}              \label{eq:thin1}
R^iE_r^*W = E_{r+i}^*W  \qquad \qquad  (0\le i\le d).
\end{equation}
\end{lemma}
\begin{proof}
Similar to the proof of \cite[Lemma 2.7]{Collins}.
%In  \eqref{eq:thin1} the inclusion  $R^iE_r^*W \subseteq E_{r+i}^*W$
%follows from Lemma \ref{lem:QnAction} since $R^i \in Q_i$.
%By the definition of endpoint and diameter and since $W$ is thin,
%$\dim E_{r+i}^*W = 1$. 
%In order to show that $R^iE_r^*W = E_{r+i}^*W$ it is therefore sufficient to
%show that the subspace $R^iE_r^*W$ is nonzero for all $i$. 
%We show this by induction. The claim is trivially true for $i=0$. Suppose
% $R^jE_r^*W = E_{r+j}^*W$ for $0 \le j < i$ and $R^iE_r^*W=0$.
%Define  $U=E_{r}^*W+ \dots + E_{r+i-1}^*W$. 
%Then by Lemma \ref{lem:LFRsubspace} and  \eqref{eq:EiEj}  the subspace $U$
%is a $T$-module. Now   $U$ is a nontrivial $T$-module that is contained in $W$. 
%A contradiction.
\end{proof}

\begin{lemma}    \label{lemma:thin2}
Let $W$ denote a thin irreducible $T$-module with endpoint $r$ 
and diameter $d$. Pick  a nonzero  $u \in E_r^*W$. Then: 
\begin{itemize}
\item[{\rm (i)}]  for $0 \le i  \le d$, $R^i u$ is a basis for $E_{r+i}^*W$;
\item[{\rm (ii)}] the vectors $\{R^i u\}_{i=0}^d$ form a basis for $W$.
\end{itemize}            
\end{lemma}
\begin{proof}
(i)  By Lemma \ref{lemma:thin1} and since $\dim \, E_{r+i}^*W = 1$ for  $0 \le i  \le d$.
\smallskip

\noindent
(ii)  Use (i) and \eqref{eq:Wsum}.
\end{proof}

\begin{definition}                              \label{def:thinbasis}
The basis in Lemma   \ref{lemma:thin2}(ii) is called \emph{standard}.
\end{definition}

\begin{lemma}    \label{lemma:thin3}
Let $W$ denote a thin irreducible $T$-module with  diameter $d$. 
Let $\{v_i\}_{i=0}^d$ denote a standard basis for $W$.
Then  $Rv_i=v_{i+1}$ $(0 \le i \le d-1)$ and $Rv_d=0$.
\end{lemma}
\begin{proof}
By construction.
\end{proof}

\begin{lemma}    \label{lemma:thin4}
Let $W$ denote a thin irreducible $T$-module with  diameter $d$. 
Then there exist scalars $\{a_i(W)\}_{i=0}^d$, $\{x_i(W)\}_{i=1}^d$
in $\R$ that satisfy the following. 
For any standard basis $\{v_i\}_{i=0}^d$ of $W$,
\begin{align}              
Fv_i &= a_i(W)v_i  \qquad \qquad  (0\le i\le d),   \label{eq:thin4a}\\
Lv_i &= x_i(W)v_{i-1}  \qquad \qquad  (1\le i\le d).   \label{eq:thin4b}
\end{align}
\end{lemma}
\begin{proof}
By Lemma \ref{lem:LFRsubspace} and since $F,L \in T$.
%By Lemma \ref{lemma:thin1} and Definition \ref{def:thinbasis}, $v_i \in E_{r+i}^*$. 
%Consequently $Fv_i \in E_{r+i}^*W$ and $Lv_i \in E_{r+i-1}^*W$ by
%by Lemma \ref{lem:LFRsubspace}. 
%Lines  \eqref{eq:thin4a} and  \eqref{eq:thin4a} now hold since $\dim E_{r+i}^*W = 0$.
\end{proof}

\begin{proposition}    \label{prop:thin5}
Let $U,W$ denote  thin irreducible $T$-modules. 
Then the following are equivalent:
\begin{itemize}
\item[{\rm (i)}]  $U$ and $W$ are quasi-isomorphic;
\item[{\rm (ii)}]  $U$ and $W$  have the same diameter $d$  and    %\label{eq:thin5a}
\begin{align*}
a_i(U) &= a_i(W)   \qquad \qquad  (0\le i\le d), \\  %\label{eq:thin5b}
x_i(U) &= x_i(W)   \qquad \qquad  (1\le i\le d). %\label{eq:thin5c}
\end{align*}
\end{itemize}            
\end{proposition}
\begin{proof}
Use Definition  \ref{def:quasiiso}.
\end{proof}

%-----------------------------11se-----------------------------------
\section{Examples}

In this section we  discuss two examples. 
The first example concerns the Hamming graphs \cite[p. 261]{BCN}. 
The second example concerns the bipartite dual polar graphs \cite[p. 274]{BCN}.  
We show that for the first example $Q=T$, and for the second
example $Q \ne T$. 

We now recall the Hamming graphs.
Fix  integers    $D \ge 1$ and $N \ge 2$.
For the \emph{Hamming graph} $H(D,N)$ the vertex set consists of
the $D$-tuples of elements from $\{1,\dots,N\}$.  Two vertices  of $H(D,N)$
are adjacent whenever they differ in precisely one coordinate.
The graph $H(D,N)$ is distance-regular %\cite[p. 126??]{BCN} 
and  $Q$-polynomial \cite[Section 9.2]{BCN}.

%We will need the following properties of $H(d,q)$, see  \cite[Theorem 9.2.1]{BCN}.

%\begin{theorem}   \label{thm:Hammingeigenvalues}
%Graph $H(D,q)$ is a $Q$-polynomial distance-regular graph.
%The eigenvalues and the dual eigenvalues of $H(D,q)$ are given by
%$$
%\theta_i=\theta_i^*=q(D-i)-D, \ \ \   \ (0 \le j \le D).
%$$
%\end{theorem}

%For the rest of this subsection let $\G = H(D,q)$.
%Let $A$ denote the adjacency matrix of $\G$.
%Fix  $x \in V(\G)$. Let $T=T(x)$ denote the subconstituent
%algebra  of $\G$.
% and let $Q$ denote the quantum adjacency algebra
%Let  $\{E_i^*\}_{i=0}^D$ denote the dual idempotents of $\G$ with respect to $x$
%and let $L,F,R$ be as in the definition \ref{def:LFR}.
%Let $\{\theta_i\}_{i=0}^D$ and $\{\theta_i^*\}_{i=0}^D$ denote
%the corresponding eigenvalue and dual eigenvalue sequences.

%Some properties of the irreducible modules of the subconstituent algebra 
%of $\G$ are described in  the following lemmas.
%We will need the following property of the irreducible modules of
%% subconstituent algebras of Hamming graphs,  
%$T$, see  \cite[Example 6.1 (13)]{Talg3}.

%\begin{lemma} \label{lem:HammingWrt}
%Let $W$ be an irreducible $T$-module.
%% with diameter $d$, endpoint $r$ and dual endpoint $t$. 
%Then  the endpoint and the dual endpoint of $W$ are equal.
%\end{lemma}

\begin{proposition}
%Let $Q$ denote the quantum adjacency algebra of $\G=H(D,q)$. 
Fix  a vertex $x$ of $H(D,N)$.
Let $Q=Q(x)$  and  $T=T(x)$.
Then $Q=T$.
\end{proposition}
\begin{proof}
Define $A^*=\sum_{i=0}^D\ \theta_i^*E_i^*$,
where $\theta_i^* = (N-1)(D-i)-i$  for $0 \le i \le D$.
%where $\theta_i^* = N(D-i)-D, \ \ \   \ (0 \le j \le D).
In \cite{Talg}, $A^*$ is called the dual adjacency matrix.
Observe that $\{\theta_i^*\}_{i=0}^D$ are mutually distinct.
Therefore $A^*$ generates $M^*$.
Consequently $A,A^*$ generate $T$. 
By construction $A \in Q$.
By combinatorial counting we obtain
$$
A^*=F+LR-RL.
$$
So $A^* \in Q$. By these comments $Q=T$.
\end{proof}

We now turn our attention to the bipartite dual polar graphs.
Fix a prime power $q$ and an integer   $D \ge 2$.
The bipartite dual polar graph $D_D(q)$ is described in \cite[Section 9.4]{BCN}; 
see also \cite{Miklavic,Talg3,Worawann}.
This graph is distance-regular %\cite[p. 126 ??]{BCN} 
and  $Q$-polynomial \cite[Section 9.4]{BCN}.

\begin{proposition}
Fix  a vertex $x$ of $D_D(q)$.
Let $Q=Q(x)$  and  $T=T(x)$.
Then $Q \ne T$. 
\end{proposition}
\begin{proof}
The irreducible $T$-modules are described in \cite{Caughman}.
By \cite[Lemma 9.2]{Caughman} these modules are thin.
By \cite[Theorem 15.6]{Caughman}, up to isomorphism there exists
a unique irreducible $T$-module $U$ with endpoint 1 and diameter $D-2$.
By \cite[Theorem 15.6]{Caughman}, up to isomorphism there exists
a unique irreducible $T$-module $W$ with endpoint 2 and diameter $D-2$.
By \cite[p. 200]{Talg3},   $a_i(U)=a_i(W)=0$ for $0 \le i \le D-2$.
Also by \cite[p. 200]{Talg3},
$$
x_i(U)=x_i(W)=\frac{q^{i+1}(q^i-1)(q^{D-i-1}-1)}{(q-1)^2} \ \ \ \ \ \ \ \ \ \ (1\le i\le D-2).
$$
By these comments $U$ and $W$ satisfy the conditions
of Proposition \ref{prop:thin5}(ii). 
By Proposition \ref{prop:thin5} the $T$-modules $U$ and $W$
are quasi-isomorphic. By this and Theorem   \ref{thm:QisT}, $Q \ne T$.
\end{proof}

%-------------------------- bib ---------------------------------------

%\bigskip
%
%\noindent
%Paul Terwilliger \\
%Department of Mathematics, University of Wisconsin \\
%480 Lincoln Drive, Madison, WI 53706-1388 USA \\
%email: \texttt{terwilli@math.wisc.edu}
%
%\bigskip
%
%\noindent
%Arjana \v Zitnik,\\
%Faculty of Mathematics and Physics, University of Ljubljana and \\
%Institute of Mathematics, Physics and Mechanics,\\
%Jadranska 19, 1000 Ljubljana, Slovenia\\
%email: \texttt{arjana.zitnik@fmf.uni-lj.si}

\end{document}